\documentclass[12pt]{amsart}
\usepackage[pdftex]{graphicx}
\usepackage{here}
\usepackage{tikz}

\usepackage{amsmath, amssymb}
\usepackage{pifont}
\usepackage{booktabs}

\numberwithin{equation}{section}
\newtheorem{theorem}{Theorem}[section]  
\newtheorem{theorem?}{``Theorem''}[section]

\newtheorem{lemma}[theorem]{Lemma}

\theoremstyle{definition}

\newtheorem{question}{Question}

\theoremstyle{remark}
\newtheorem{remark}[theorem]{Remark}  

\newcommand{\R}{{\mathbb R}}
\newcommand{\C}{{\mathbb C}}

\newcommand{\N}{{\mathbb N}}
\newcommand{\Z}{{\mathbb Z}}

\renewcommand{\a}{\alpha}
\renewcommand{\b}{\beta}
\renewcommand{\d}{\partial}

\begin{document}
\title[local zeta functions]
{
Meromorphy of local zeta functions
in smooth model cases
} 
\author{Joe Kamimoto and Toshihiro Nose}
\address{Faculty of Mathematics, Kyushu University, 
Motooka 744, Nishi-ku, Fukuoka, 819-0395, Japan} 
\email{
joe@math.kyushu-u.ac.jp}
\address{Faculty of Engineering, Fukuoka Institute of Technology, 
Wajiro-higashi 3-30-1, Higashi-ku, Fukuoka, 811-0295, Japan}
\email{ 
nose@fit.ac.jp}
\maketitle


\begin{abstract}
It is known that 
local zeta functions associated with real analytic 
functions can be analytically continued 
as meromorphic functions to the whole complex plane.
But, in the case of general ($C^{\infty}$) smooth functions, 
the meromorphic 
extension problem is not obvious.
Indeed, it has been recently shown that
there exist specific smooth functions whose 
local zeta functions have singularities different from poles.
In order to understand the situation of 
the meromorphic extension in the smooth case, 
we investigate a simple but essentially important case, 
in which the respective function is expressed as 
$u(x,y)x^a y^b +$ flat function, 
where $u(0,0)\neq 0$ and $a,b$ are nonnegative integers.
After classifying flat functions into four types, 
we precisely investigate the meromorphic extension
of local zeta functions in each cases.  
Our results show new interesting phenomena 
in one of these cases. 
Actually, when $a<b$, 
local zeta functions can be 
meromorphically extended to the half-plane
${\rm Re}(s)>-1/a$ and
their poles on the half-plane 
are contained in the set 
$\{-k/b:k\in\N \,\,{\rm with }\,\, k<b/a\}$.
\end{abstract}




\section{Introduction}

Let us consider the integrals of the form
\begin{equation}\label{eqn:1.1}
Z_f(\varphi)(s)=\int_{\R^2}|f(x,y)|^s \varphi(x,y)dxdy
\quad\,\,\, \text{for $s\in \C$,}
\end{equation}
where
$f, \varphi$ are real-valued ($C^{\infty}$) smooth functions
defined on a sufficiently small open neighborhood $U$ of the origin
in $\R^2$, 
and 
the support of $\varphi$ is contained in $U$.
Since the integrals $Z_f(\varphi)(s)$
converge locally uniformly 
on the half-plane ${\rm Re}(s)>0$,
they become holomorphic functions
there, which are sometimes called {\it local zeta functions}.
It has been known in many cases that
they can be analytically continued 
to wider regions. 
The purpose of this paper is to understand 
the analytic continuation of local zeta functions.

When $f$ is real analytic, 
the analytic continuation of local zeta functions 
have been precisely understood.
By using Hironaka's resolution of singularities \cite{Hir64}, 
it was shown in \cite{BeG69}, \cite{Ati70},  etc. that
$Z_f(\varphi)(s)$ can be analytically continued as meromorphic functions 
to the {\it whole} complex plane
and their poles are contained in finitely many arithmetic progressions 
consist of negative rational numbers.
More precisely, 
Varchenko \cite{Var76} investigates the exact location of poles
of local zeta functions
by using the theory of toric varieties 
based on the {\it Newton polyhedron} of $f$
under some nondegeneracy condition. 
We remark that the above-mentioned results also hold in 
general dimensional case. 

On the other hand,
the smooth case (without real analyticity assumption) 
is not so well known.  
It is known in \cite{KaN16jmst} that 
the above result due to Varchenko \cite{Var76} 
can be naturally generalized when $f$ belongs to
a certain wide class of smooth functions containing 
the real analytic functions (see Appendix A.4).
But the problem of meromorphic extension of 
$Z_f(\varphi)$ 
is not obvious in the general smooth case. 
Indeed, it was observed in \cite{Gre06}, \cite{KaN19} that  
$Z_f(\varphi)(s)$ has a singularity different from a pole
in the case of a specific non-real analytic $f$ (see (\ref{eqn:2.9})). 

In this paper, we precisely investigate the case when $f$ 
can be expressed in the form:
\begin{equation}\label{eqn:1.2}
f(x,y)=u(x,y)x^a y^b+ \mbox{(flat function at the origin)},
\end{equation}
where $a,b\in\Z_+$ 
and $u$ is a smooth function defined near the origin with
$u(0,0)\neq 0$.
(For a smooth function $g$ defined near the origin,  
we say $g$ is {\it flat} at the origin
if all the derivatives of $g$ vanish at the origin.)
This case is a simple generalization of the above 
non-real analytic function $f$ 
in \cite{Gre06}, \cite{KaN19}. 
Recalling that 
the monomial case 
is essentially important in the real analytic case
(see \cite{Var76}, \cite{AGV88}) and noticing that
the Newton polyhedron of $f$ in 
(\ref{eqn:1.2}) is a simple form:
$\{(\alpha,\beta)\in\R^2:\alpha\geq a,\,\, 
\beta\geq b\}$ (see Figure~1, below), 
the above case (\ref{eqn:1.2}) 
might be considered as a natural model 
in the smooth case.  

It has been recognized (c.f. \cite{AGV88}) that
the analytic continuation of local zeta functions
is deeply related to 
the behavior at infinity of 
oscillatory integrals of the form 
\begin{equation*}
I_f(\varphi)(\tau):=
\int_{\R^2}e^{i \tau f(x,y)}\varphi(x,y)dxdy 
\quad \mbox{for $\tau>0$},
\end{equation*}
where  $f$, $\varphi$ are the same as in (\ref{eqn:1.1}). 
The investigation of the behavior 
of oscillatory integrals has similarities to 
the analytic continuation of local zeta functions. 
In fact, 
the case when the phase $f$ is real analytic   
is well understood and, moreover, 
there have been results under some conditions in smooth case, 
which are direct generalization of 
those of the real analytic case
(\cite{Gre09}, \cite{IKM10}, 
\cite{IkM11jfaa}, \cite{KaN16jmst}, 
\cite{Gil18}, etc.).
But, it is shown in \cite{IoS97}, \cite{KaN17} that 
when the phase contains a flat function, 
the behavior of $I_f(\varphi)$ may have a different pattern  
from that in the real analytic case. 
It is also interesting to consider 
how flat functions in the phase affect
the behavior of $I_f(\varphi)(\tau)$, which 
is analogous to that of 
analytic continuation of local zeta functions
in this paper.

This paper is organized as follows. 
After explaining earlier work 
about analytic continuation of local zeta functions 
and related issues in Section 2, 
we state a main result in Section 3. 
Sections 4--6 are devoted to 
the proof of the main results. 
Our investigation of this paper can be explained    
without language of Newton polyhedra and 
related words, which are important 
in singularity theory. 
The meaning of our analysis is clearer
from the viewpoint of the geometry 
of Newton polyhedra. 
In the Appendix, 
the definitions of
Newton polyhedra and related important words 
are given 
and 
our studies are explained from these points of view.

{\it Notation and symbols.}

\begin{itemize}
\item 
We denote $\R_+:=\{x\in \R:x\geq 0\}$
and $\Z_+:=\{x\in \Z:x\geq 0\}$.
\item
In this paper, Newton polyhedra appear in many situations  
(see Appendix~A.1).
We use coordinates $(\alpha,\beta)$ for points in the plane
containing the Newton polyhedron 
in order to distinguish this plane from the $(x,y)$-plane. 
\end{itemize}


\section{Known results and description of the problems}

In this paper, we always assume that $f$ satisfies 
\begin{equation}\label{eqn:2.1}
f(0,0)=0 \quad \mbox{and} \quad \nabla f (0,0)=(0,0).
\end{equation}
Unless (\ref{eqn:2.1}) is satisfied,  
every problem treated in this paper is easy.
As for $\varphi\in C_0^{\infty}(U)$, 
we sometimes give the following condition:
\begin{equation}\label{eqn:2.2}
\varphi(0,0)>0 \mbox{\,\,\, and \,\,\,} \varphi\geq 0
\mbox{ on $U$}.
\end{equation}
In order to investigate the analytic continuation of 
local zeta functions, 
we only use the half-plane of the form:
${\rm Re}(s)>-\rho$
with $\rho>0$.
This is the reason why we observe the situation of analytic continuation
through the integrability 
of the integral (\ref{eqn:1.1}). 
(Of course,
many kinds of regions should be considered
in the future.)

\subsubsection{Holomorphic continuation}
First, let us consider the following quantities:
\begin{equation}\label{eqn:2.3}
h_0(f,\varphi):=\sup\left\{\rho>0: 
\begin{array}{l} 
\mbox{The domain in which $Z_f(\varphi)$ can} \\ 
\mbox{be holomorphically continued} \\
\mbox{contains the half-plane ${\rm Re}(s)>-\rho$}
\end{array}
\right\}, 
\end{equation}
\begin{equation}\label{eqn:2.4}
h_0(f):=\inf\left\{h_0(f,\varphi):
\varphi\in C_0^{\infty}(U)\right\}.
\end{equation}
We remark that 
if $\varphi$ satisfies (\ref{eqn:2.2}), 
then $h_0(f,\varphi) = h_0(f)$ holds; 
but
otherwise, 
this equality does not always hold.
Indeed,
there exists $\varphi\in C_0^{\infty}(U)$ with 
$\varphi(0,0)=0$
such that $h_0(f,\varphi)>h_0(f)$ 
(see e.g. \cite{KaN16tams}).

From the integral representation in (\ref{eqn:1.1}),
the relationship between the holomorphy and 
the convergence of the integral implies that  
the quantity $h_0(f)$ is deeply related to
the following famous index:
\begin{equation}\label{eqn:2.5}
c_0(f):=\sup
\left\{\mu>0:
\begin{array}{l} 
\mbox{there exists an open neighborhood $V$ of} \\
\mbox{the origin in $U$ such that $|f|^{-\mu}\in L^1(V)$}
\end{array}
\right\},
\end{equation}
which is called 
{\it log canonical threshold} or {\it critical integrability index}
and has been deeply understood 
from various kinds of viewpoints. 
The equality $h_0(f)=c_0(f)$ always holds. 
In fact, 
the inequality $h_0(f)\geq c_0(f)$ is obvious; 
while the opposite inequality can be easily seen
by Theorem~5.1 in \cite{KaN19}.
In the real analytic case, 
since all the singularities of the extended $Z_f(\varphi)$ 
are poles on the real axis, 
the leading pole exists at $s=-h_0(f,\varphi)$.
In the seminal work of Varchenko \cite{Var76},
when $f$ is real analytic and satisfies some conditions,  
$h_0(f)$ can be expressed by using 
the {\it Newton polyhedron} of $f$ as
\begin{equation}\label{eqn:2.6}
h_0(f)=1/d(f),
\end{equation}
where $d(f)$ is the {\it Newton distance} of $f$
(see Appendix A.2).
More detailed investigations into meromorphic continuation 
of $Z_f(\varphi)$ in various situations are in 
\cite{DeS89}, \cite{DeS92}, \cite{DNS05}, \cite{OkT13}, \cite{AGL19}, 
etc. 
An interesting work \cite{CGP13} treating the equality 
$c_0(f)=1/d(f)$ is from another approach.
We remark that these results treat the general dimensional case. 
In the same paper \cite{Var76}, 
Varchenko more deeply investigated the two-dimensional case. 
Indeed, without any assumption, he shows that 
the equality (\ref{eqn:2.6}) holds for real analytic $f$
in {\it adapted coordinates}. 
Here the definition of adapted coordinates is given
in Appendix A.3, below. 
These coordinates are important in the study of oscillatory integrals 
and their existence is shown in two dimensions 
in \cite{Var76}, \cite{PSS99}, \cite{IkM11tams}, etc.
More generally, let us consider the smooth case. 
M. Greenblatt \cite{Gre06} obtains a sharp result 
which generalizes Varchenko's two-dimensional result.
\begin{theorem}[Greenblatt \cite{Gre06}]
When $f$ is a nonflat smooth function defined on $U$, 
the equation $c_0(f)=1/d(f)$
holds in adapted coordinates.
\end{theorem}

From the above result, 
holomorphic extension issue is well understood 
even in the smooth case. 
On the other hand, the situation of
the meromorphic extension 
is quite different from the holomorphic one
and has not been so well known. 

\subsubsection{Meromorphic continuation}

Corresponding to (\ref{eqn:2.3}), (\ref{eqn:2.4})
in the holomorphic continuation case,
we analogously define the following quantities:
\begin{equation}\label{eqn:2.7}
m_0(f,\varphi):=\sup\left\{\rho>0: 
\begin{array}{l} 
\mbox{The domain in which $Z_f(\varphi)$ can} \\ 
\mbox{be meromorphically continued} \\
\mbox{contains the half-plane ${\rm Re}(s)>-\rho$}
\end{array}
\right\}, 
\end{equation}
\begin{equation}\label{eqn:2.8}
m_0(f):=\inf\left\{m_0(f,\varphi):
\varphi\in C_0^{\infty}(U)\right\}.
\end{equation}
It is obvious that $h_0(f)\leq m_0(f)\leq m_0(f,\varphi)$ and 
$h_0(f,\varphi)\leq m_0(f,\varphi)$ always holds.
As mentioned in the Introduction, 
if $f$ is real analytic, then 
$m_0(f)=\infty$ holds; while
there exist specific non-real analytic functions $f$ such that
$m_0(f)<\infty$. 
Indeed, it was shown in \cite{KaN19} (see also \cite{Gre06}) 
that when
\begin{equation}\label{eqn:2.9}
f(x,y)=x^a y^b+x^ay^{b-q}e^{-1/|x|^p},
\end{equation}
where $a, b, q\in\Z_+$ satisfy that
$a<b$, \, $b\geq 2$, \, $1\leq q\leq b$, \,
$q$ is even,
$p>0$
and
$\varphi$ satisfies the condition \eqref{eqn:2.2}, 
$Z_f(\varphi)(s)$ has a non-polar singularity at 
$s=-1/b$,
which implies $m_0(f)=1/b$.
Note that $d(f)=b$ in this case. 
At present, properties of the singularity 
at $s=-1/b$ are not well understood
(see \cite{KaN19} for the details).

\begin{question}
For a given smooth functions $f$, 
determine (or estimate) the value of $m_0(f)$.
\end{question}

In this paper, we consider the above question 
in the case when $f$ has the form (\ref{eqn:1.2})
which is a natural generalization of (\ref{eqn:2.9}). 


\section{Main results}

In this section, let $f$ be expressed as in (\ref{eqn:1.2})
on some small open neighborhood $U$ of the origin.
Without loss of generality, we assume that 
$a, b\in\Z_+$ in (\ref{eqn:1.2}) 
satisfy $a\leq b$ and $u(0,0)>0$. 
Moreover, we always assume that $b\geq 1$.
In fact, when $a=b=0$, $Z_f(\varphi)$ becomes an entire function
if the support of $\varphi$ is sufficiently small. 

It is easy to check the following Newton data of $f$ in (\ref{eqn:1.2}):
\begin{itemize}
\item The Newton polyhedron of $f$:   
$\Gamma_+(f)=
\{(\alpha,\beta)\in\R_+^n:\alpha\geq a, \, \beta\geq b \}$.
\item The Newton distance of $f$: $d(f)=b$.
\item The function $f$ in (\ref{eqn:1.2}) is expressed 
in an {\it adapted coordinate}.
\end{itemize}

\begin{lemma}
If $U$ is sufficiently small, then
$f$ in (\ref{eqn:1.2}) can be expressed on $U$ 
as one of the following four forms. 
\begin{enumerate}
\item[(A)]
$f(x,y)=
v(x,y)x^a y^b,$
\item[(B)]
$f(x,y)=
v(x,y)x^a y^b 
+ g(x,y),$
\item[(C)]
$f(x,y)=
v(x,y)x^a y^b 
+ h(x,y),$
\item[(D)]
$f(x,y)=
v(x,y)x^a y^b 
+ g(x,y)+h(x,y).$
\end{enumerate}
where $v,h,g$ are smooth functions defined on $U$ 
satisfying the following properties:
\begin{itemize}
\item $v(0,0)=u(0,0)\neq 0$.
\item 
$h, g$ are {\it non-zero} flat functions admitting the forms:
\begin{equation}\label{eqn:3.1}
g(x,y)=\sum_{j=0}^{b-1}y^{j}g_j(x)
\quad
\mbox{ and } \quad
h(x,y)=\sum_{j=0}^{a-1}x^{j}h_j(y),
\end{equation}
where $h_j, g_j$ are flat at the origin.
\end{itemize}
\end{lemma}

\begin{proof}
For simplicity,
we use the following symbol:
For $(\alpha,\beta)\in \Z_+^2$,
\begin{equation*}
f^{\langle \alpha,\beta \rangle}(x,y):=
\frac{\d^{\alpha+\beta}f}{\d x^{\alpha}\d y^{\beta}}(x,y).
\end{equation*}
The Taylor formula implies 
\begin{equation*}
\begin{split}
f(x,y)
&=
\sum_{\b=0}^{b-1}y^{\b}
A_{\beta}(x)
+
\sum_{\a=0}^{a-1}x^{\a}
B_{\alpha}(y)
+
x^{a}y^{b}C(x,y),
\end{split}
\end{equation*}
where
\begin{equation*}
\begin{split}
&A_{\beta}(x)
:=\frac{x^a}{(a-1) ! \b!}\int_{0}^{1}(1-t)^{a-1}
f^{\langle a,\b \rangle}(tx,0)dt
 \mbox{\quad for } \b\in\{0,\ldots, b-1\}, \\
&B_{\alpha}(y)
:=\frac{y^b}{\alpha ! (b-1)!}
\int_{0}^{1}(1-t)^{b-1}
f^{\langle \a, b \rangle}(0,ty)dt
\mbox{\quad for }  \a\in \{0,\ldots,a-1\},\\
&C(x,y)
:=\frac{1}{(a-1)! (b-1)!}\int_{0}^{1}\int_{0}^{1}
(1-t_1)^{a-1}(1-t_2)^{b-1}
f^{\langle a,b \rangle}(t_1x,t_2y)
dt_1 dt_2.
\end{split}
\end{equation*}
Since $f^{\langle a,\b \rangle}(\cdot,0)$ 
is flat at the origin for $\b\in\{0,\ldots, b-1\}$, 
so is $A_{\b}$ 
for $\b\in\{0,\ldots, b-1\}$.
The flatness of $B_{\a}$ is similarly shown
for $\a\in\{0,\ldots,a-1\}$. 
An easy computation gives 
$C(0,0)=u(0,0)$.
Putting $g_j(x):=A_{j}(x)$, 
$h_j(y):=B_j(y)$ and $v(x,y):=C(x,y)$, 
we can obtain the lemma. 
\end{proof}

\begin{remark}
(1)\quad 
The example (\ref{eqn:2.9}) mentioned in 
Section~2 belongs to the case (B).

(2)\quad 
In Lemma~3.1,
$v$ does not always equal $u$ in (\ref{eqn:1.2}) 
on $U$ and, moreover, 
in the cases (B), (C), (D), $v$ cannot always be replaced 
by $v\equiv 1$ by using coordinate changes. 

(3)\quad 
It is easy to see the following equivalences. 
\begin{enumerate}
\item When $a=b$, (B) $\Longleftrightarrow$ (C).
\item When $a=0$, (A) $\Longleftrightarrow$ (C) and 
(B) $\Longleftrightarrow$ (D).
\end{enumerate}
The equivalence in (i) means that the roles of the $x$ and 
$y$ variables can be switched. 
\end{remark}

Let us observe the above classification 
from the viewpoint of the geometry of Newton polyhedra.
Of course, the existence of the flat functions $h,g$
give no influence on the shape of the Newton polyhedron 
$\Gamma_+(f)$. 
But, in Figure~1, 
we forcibly draw their influence by adding 
the points $(0,\infty)$, $(\infty,0)$ 
(see also Appendix A.4).

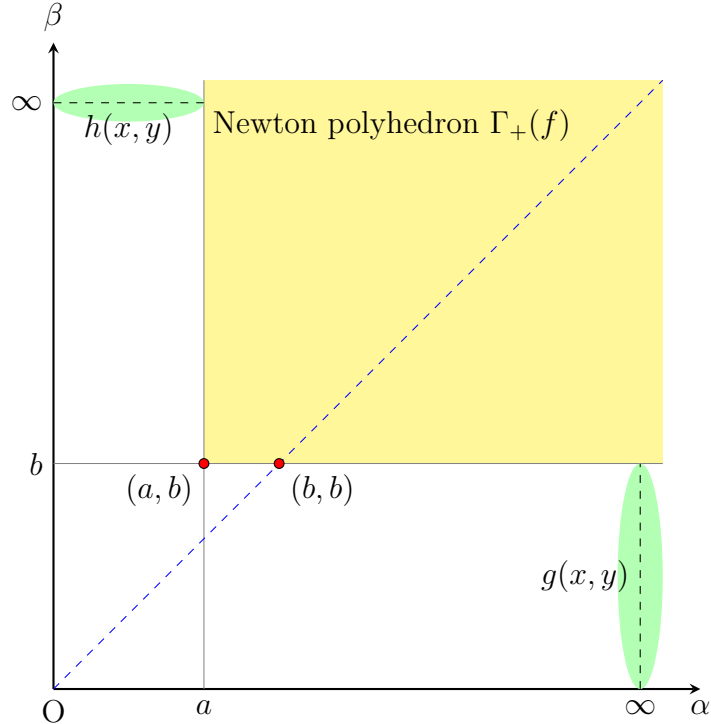
\begin{figure}[H]

\begin{tikzpicture}
\fill[yellow!50!] (-2,-1) -- (-2,4.1) -- (4.1,4.1) -- (4.1,-1);
   \draw [thick, -stealth](-4,-4)--(4.6,-4) node [anchor=north]{$\alpha$};
   \draw [thick, -stealth](-4,-4)--(-4,4.6) node [anchor=south]{$\beta$};

   \coordinate (S) at (0.5,3) node at (S) 
[label=above: Newton polyhedron $\Gamma_+(f)$ ]{};
   \node (S) at (-1.5,0) {};
   \draw [help lines] (-4,-1)--(4.1,-1);  
   \draw [help lines] (-2,-4)--(-2,4.1);  
   \draw [blue, dashed] (-4,-4) -- (4.1,4.1);
  \coordinate (a) at (-2,-4) node at (a) [below] {$a$};
  \coordinate (b) at (-4,-1) node at (b) [left] {$b$};
  \coordinate (i1) at (3.8,-4) node at (i1) [below] {$\infty$};
  \coordinate (i2) at (-4,3.8) node at (i2) [left] {$\infty$};
  \coordinate (c) at (-1,-1);
\filldraw[fill=red] (c) circle[radius=0.65mm] node [below right]{$(b,b)$};
  \coordinate (d) at (-2,-1);
\filldraw[fill=red] (d) circle[radius=0.65mm] node [below left] {$(a,b)$};
\fill[green!30] (3.8,-2.5) ellipse (0.3 and 1.5);
\fill[green!30] (-3,3.8) ellipse (1 and 0.25);
   \draw [dashed] (3.8,-4) -- (3.8,-1);
   \draw [dashed] (-4,3.8) -- (-2,3.8);
\coordinate (g) at (3.8,-2.5) node at (g) [left] {$g(x,y)$};
\coordinate (h) at (-3,3.8) node at (h) [below] {$h(x,y)$};
   \node [anchor=north] at (-4,-4) {O};
\end{tikzpicture} 

\caption[The Newton polyhedron $\Gamma_+(f)$]
{The Newton polyhedron $\Gamma_+(f)$}
\end{figure}

Let us investigate the quantities 
$h_0(f)$, $m_0(f)$ with $f$ in (\ref{eqn:1.2})
in each of the above cases. 
As explained in the previous section, 
Theorem~2.1 implies $h_0(f)=1/b$ in all the cases. 
Now, 
let us consider the value of $m_0(f)$.

In the case (A), it is easy to see that $m_0(f)=\infty$
(see also Appendix~A.5).
In the cases (B), (C), (D),
it follows from Theorem~2.1 that 
the estimate $m_0(f)\geq 1/b$ always holds. 
This estimate is optimal in the case (B).  
(Recall that $m_0(f)=1/b$ holds when $f$ is as in (\ref{eqn:2.9}).)
In the case (C), we see a new phenomenon of 
meromorphic continuation, 
which is a main theorem of this paper.
\begin{theorem}\label{thm:1.2}
Let $a>0$ and 
let $f$ be expressed as in the case (C) in Lemma~3.1 on $U$.
If the support of $\varphi$ is sufficiently small,
then $Z_f(\varphi)(s)$
can be analytically continued as a meromorphic function
to the half-plane ${\rm Re}(s)>-1/a$.
Moreover, when $a<b$,
its poles on the region ${\rm Re}(s)>-1/a$ 
are contained in the set 
$\{-k/b:k\in \N \,\,{\rm with}\,\, k<b/a\}$.
In particular, $m_0(f)\geq 1/a$ holds.
\end{theorem}

\begin{remark}
(1)\quad In the forthcoming paper \cite{KaN20}, 
we will show that
there exists a specific function $f$ of the form (C) such that
$Z_f(\varphi)$ has the non-polar singularity
at $s=-1/a$, which implies $m_0(f)=1/a$.
Therefore, 
the estimate $m_0(f)\geq 1/a$ in Theorem~3.3 is optimal.
Moreover, 
the optimality of
$m_0(f)\geq 1/b$ in the case (D) will also be shown
in the same paper.

(2)\quad 
From Lemma~6.2 (i) below playing essential roles in 
the proof of the above theorem, 
the readers might wonder if 
$\{-j/a:j\in\N\}$ is also contained in 
the set of candidate poles of $Z_f(\varphi)$. 
Since properties of 
the singularity of $Z_f(\varphi)$ on the line 
 ${\rm Re}(s)=-1/a$ 
have not been well-understood at present,  
$Z_f(\varphi)$ can be ragarded 
as a meromorphic function 
{\it only} on the region ${\rm Re}(s)>-1/a$
in general.
Noticing that $\{-j/a:j\in\N\}$ and 
$\{-k/b:k\geq b/a\}$ 
are outside of the region ${\rm Re}(s)>-1/a$, 
we see that 
$\{-k/b:1\leq k < b/a\}$ only appears in the theorem. 
\end{remark}

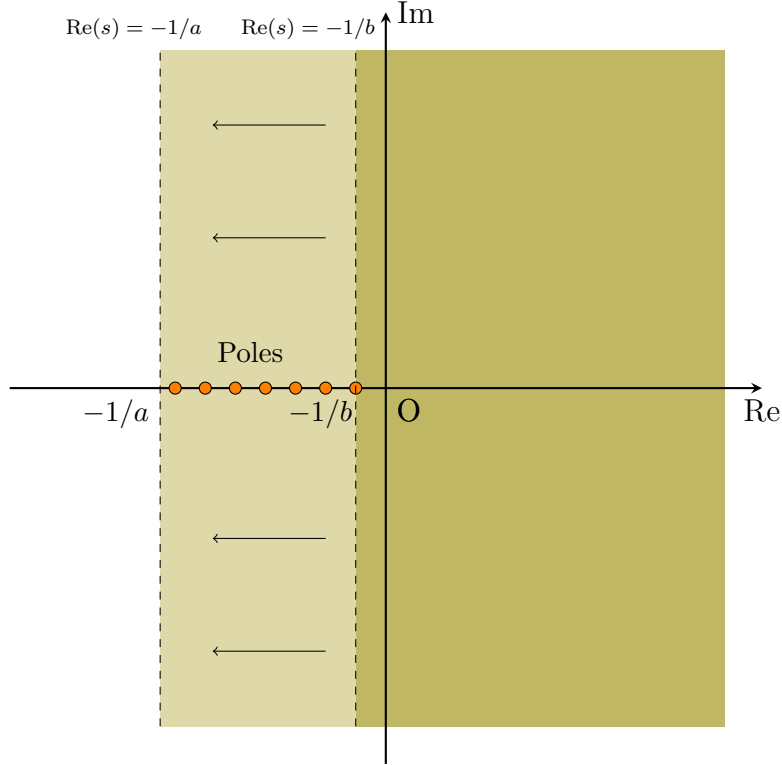
\begin{figure}[H]

\begin{tikzpicture}
   \fill[olive!60!] (-0.4,-4.5) -- (-0.4,4.5) -- (4.5,4.5) -- (4.5,-4.5);
 \fill[olive!30!] (-0.4,-4.5) -- (-0.4,4.5) -- (-3,4.5) -- (-3,-4.5);
   \draw [thick, -stealth](-5,0)--(5,0) node [anchor=north]{Re};
   \draw [thick, -stealth](0,-5)--(0,5) node [anchor=west]{Im};
   \node [anchor=north west] at (0,0) {O};
   \node [anchor=north west] at (0,0) {O};
\coordinate (P1) at (-0.4,0);
\coordinate (P2) at (-0.8,0);
\coordinate (P3) at (-1.2,0);
\coordinate (P4) at (-1.6,0);
\coordinate (P5) at (-2.0,0);
\coordinate (P6) at (-2.4,0);
\coordinate (P7) at (-2.8,0);
\node [anchor=north east] at (-0.3,0) {{\small$-1/b$}};
\node [anchor=north east] at (-3,0) {{\small$-1/a$}};
\filldraw[fill=orange] (P1) circle[radius=0.8mm]; 
\filldraw[fill=orange] (P2) circle[radius=0.8mm];
\filldraw[fill=orange] (P3) circle[radius=0.8mm];
\filldraw[fill=orange] (P4) circle[radius=0.8mm];
\filldraw[fill=orange] (P5) circle[radius=0.8mm];
\filldraw[fill=orange] (P6) circle[radius=0.8mm];
\filldraw[fill=orange] (P7) circle[radius=0.8mm];
   \draw [dashed] (-3,-4.5)--(-3,4.5); 
\node [anchor=south east] at (-2.3,4.5) {{\tiny ${\rm Re}(s)=-1/a$}};
   \draw [dashed] (-0.4,-4.5)--(-0.4,4.5); 
\node [anchor=south east] at (0,4.5) {{\tiny ${\rm Re}(s)=-1/b$}};

\node [above] at (-1.8,0.2) {{\small Poles}};

\draw [->, color=black] (-0.8,2)--(-2.3,2);
\draw [->, color=black] (-0.8,3.5)--(-2.3,3.5);
\draw [->, color=black] (-0.8,-2)--(-2.3,-2);
\draw [->, color=black] (-0.8,-3.5)--(-2.3,-3.5);
\end{tikzpicture}

\caption[]{Meromorphic continuation in the case (C)}
\end{figure}

\begin{table}[H]
\caption[]{The values of $h_0(f)$ and $m_0(f)$.}
\begin{center}
\begin{tabular}{r|cccc} \toprule
 & (A) & (B) & (C) & (D) \\  \midrule 
$h_0(f)$ & $1/b$ & $1/b$ & $1/b$ & $1/b$ \\
$m_0(f)$  & $\infty$ & $\geq 1/b$ & $\geq 1/a$ & $\geq 1/b$ 
\\ \bottomrule
\end{tabular}
\end{center}
\end{table}

Putting the above mentioned results together, 
one has Table~1. 
We remark that Table~1 establishes for $a>0$. 
But, by regarding $1/0$ as $\infty$ and 
recalling Remark 3.2 (ii), 
this restriction is not needed.

It should be expected that the inequalities ``$\geq$'' 
in Table~1 can be removed. 
In other words, 
the following question is naturally raised.

\begin{question}
Do the following equalities hold? 
\begin{enumerate}
\item $m_0(f)=1/b$ for {\it all} $f$ satisfying (B) or (D);
\item $m_0(f)=1/a$ for {\it all} $f$ satisfying (C) with $a>0$.
\end{enumerate}
\end{question}

As mentioned above, some specific cases showing the above equalities 
are known but they are very special. 
It seems to be difficult to generally show the equalities
from our method in this paper.
Indeed, 
after some kind of approximation, 
we apply a van der Corput-type lemma (Lemma~4.5, below). 
Although this process is available for general smooth functions,  
this approximation is an obstacle to see the
situation of behavior of local zeta functions
near the line ${\rm Re}(s)=-1/a$ or $-1/b$.




\section{Auxiliary lemmas}


\subsection{Meromorphy of one-dimensional model}

The following lemma is essentially known
(see \cite{GeS64}, \cite{BeG69}, \cite{AGV88}, etc.).
Since we will use
not only the result but also an idea of its proof
in the later computation, 
we give a complete proof here.

\begin{lemma}\label{lem:2.2}
Let $A,B$ be integers with $A>0$, $B\geq 0$ and
let $\psi(u;s)$ be a complex-valued function defined on $[0,r] \times \C$, 
where $r>0$.
We assume that 
\begin{enumerate}
\item[(a)]
$\psi(\cdot;s)$ is smooth on $[0,r]$ 
for all $s\in \C$;
\item[(b)] 
$\dfrac{\d^{\a} \psi}{\d u^{\a}}(u;\cdot)$ is
an entire function on $\C$ for all $u\in [0,r]$ and $\a\in \Z_+$.
\end{enumerate}
Let
\begin{equation}\label{eqn:4.1}
L(s):=\int_{0}^{r}u^{A s+B}\psi(u;s)du.
\end{equation}
Then the following hold.
\begin{enumerate}
\item
The integral
$L(s)$ 
becomes a holomorphic function on the half-plane ${\rm Re}(s)>-(B+1)/A$.
\item
The integral
$L(s)$ can be analytically continued as a meromorphic function
to the whole complex plane. 
Moreover, its poles are simple and they are contained in the set 
$\{-(B+j)/A: j\in \N\}$.
\end{enumerate}
\end{lemma}

\begin{proof}
(i) \quad 
Since $L(s)$ locally uniformly converges 
on the half-plane ${\rm Re}(s)>-(B+1)/A$,
the assumption and the Lebesgue convergence theorem give that
the integral becomes a holomorphic function there.

\vspace{.5 em}

(ii) \quad
Let $N$ be an arbitrary natural number.
The Taylor formula implies
\begin{equation}\label{eqn:4.2}
\psi(u;s)=\sum_{\a=0}^{N}\frac{1}{\a !}\frac{\d^{\a} \psi}{\d u^{\a}}(0;s)u^{\a}+u^{N+1}R_N(u;s)
\end{equation}
with
\begin{equation*}
R_N(u;s)=
\frac{1}{N!}\int_{0}^{1}(1-t)^{N}
\frac{\d^{N+1} \psi}{\d u^{N+1}}(tu;s)dt.
\end{equation*}
Here $R_N(u;s)$ satisfies the following.    
\begin{itemize}
\item
$R_N(\cdot;s)$ is smooth on $[0,r]$ 
for all $s\in \C$.
\item
$R_N(u;\cdot)$ is
an entire function on $\C$ for all $u\in [0,r]$.
\end{itemize}.
Substituting (\ref{eqn:4.2}) into the integral (\ref{eqn:4.1}), 
we have 
\begin{equation}\label{eqn:4.3}
\begin{split}
L(s)=
\sum_{\a=0}^{N}
\frac{r^{As+B+\a+1}}{\a!(As+B+\a+1)}
\frac{\d^{\a} \psi}{\d u^{\a}}(0;s)+
\int_{0}^{r}u^{As+B+N+1}R_N(u;s)du
\end{split}
\end{equation}
on ${\rm Re}(s)>-(B+1)/A$.
From (i), the integral in (\ref{eqn:4.3}) 
becomes a holomorphic function on the half-plane ${\rm Re}(s)>-(B+N+2)/A$.
Therefore, 
$L(s)$ can be analytically continued as a meromorphic function
to the half-plane ${\rm Re}(s)>-(B+N+2)/A$.  
Moreover, all the poles of $L(s)$ are simple 
and they are contained in the set 
$\{-(B+j)/A: j=1,\ldots, N+1\}$.
Letting $N$ tend to infinity, we have the assertion.
\end{proof}

\begin{remark}
Let $\tilde{\psi}(u;s)$ be a complex-valued function defined on $\R \times \C$
satisfying that
$\tilde{\psi}(\cdot;s)$ is
a $C_0^{\infty}$ function
on $\R$ for all $s\in \C$
and
$\frac{\d^{\a} \tilde{\psi}}{\d u^{\a}}(u;\cdot)$ is
an entire function on $\C$ for all $u\in \R$ and $\a\in \Z_+$.
Then
the integral 
\begin{equation*}
\tilde{L}(s):=\int_{0}^{\infty}u^{A s+B}\tilde{\psi}(u;s)du
\end{equation*}
has the same meromorphy properties as those of $L(s)$
in Lemma~4.1.
Indeed, for $r>0$, 
$\psi(\cdot;s):=\tilde{\psi}(\cdot,s)|_{[0,r]}$
satisfies the assumptions (a), (b) of Lemma~4.1 
and 
the integral 
$\int_r^{\infty}u^{As+B}\tilde{\psi}(u;s)du$
becomes an entire function.
\end{remark}

\subsection{Meromorphy of important integrals}

The following lemma will play a useful 
role in the proof of Theorem~3.3.

\begin{lemma}
Let $a,b\in \N$ with $a \leq b$ and 
let 
\begin{equation}\label{eqn:4.4}
D:=\{(u,v)\in \R_+^2: 
v^p < u \leq R, \,\, v \leq r\},
\end{equation}
where $p\in\N$ and $r, R >0$ with $r^p\leq R$.
We assume that
$\psi(u,v;s)$ is a complex-valued function 
defined on $D \times \C$ satisfying the following.
\begin{enumerate}
\item[(a)]
$\psi(\cdot; s)$ can be smoothly extended to  
$\overline{D}$ for all $s\in \C$;
\item[(b)]
$\dfrac{\d^{\alpha+\beta}\psi}{\d u^{\alpha}\d v^{\beta}}(u,v;\cdot)$ 
is an entire function 
for all $(u,v)\in D$ and $(\alpha, \beta)\in \Z_+^2$.
\end{enumerate}
Let 
\begin{equation}\label{eqn:4.5}
H(s):=\int_{D}u^{as}v^{bs}\psi(u,v;s)dudv.
\end{equation}
Then the following hold.
\begin{enumerate}
\item
The integral $H(s)$ becomes a holomorphic function on the half-plane 
${\rm Re}(s)>-1/b$.
\item
The integral $H(s)$
can be analytically continued as a meromorphic function
to the whole complex plane
and its poles are contained in the set
\begin{equation}\label{eqn:4.6}
\left\{-\frac{j}{a},-\frac{k}{b},-\frac{p+l}{ap+b}:j,k,l\in \N\right\}.
\end{equation}
\end{enumerate}
\end{lemma}

\begin{remark}
(1) \quad
In Lemma~4.3, when the set $D$ is replaced by the set 
\begin{equation*}
\tilde{D}=\{(u,v)\in \R^2: 
0 \leq u \leq v^{p}, 0\leq v \leq r\},
\end{equation*}
the same assertions (i), (ii) hold. 

(2) \quad 
When $j=k=l=1$ in (\ref{eqn:4.6}), 
the following inequalities hold: 
$$
-\frac{1}{a}\leq-\frac{p+1}{ap+b}\leq
-\frac{1}{b}.
$$
\end{remark}

\begin{proof}[Proof of Lemma~4.3]
(i)\quad
In a similar fashion to the proof of the Lemma~4.1 (i),
it can be easily shown that $H(s)$ becomes a holomorphic function
defined on the half-plane ${\rm Re}(s)>\max\{-1/a,-1/b\}=-1/b$.
\vspace{.5 em}

(ii)\quad
Let us consider meromorphic continuation of $H(s)$.
For simplicity,
we use the following symbol:
For $(\alpha,\beta)\in \Z_+^2$,
\begin{equation*}
\psi^{\langle \alpha,\beta \rangle}(u,v;s):=
\frac{\d^{\alpha+\beta}\psi}{\d u^{\alpha}\d v^{\beta}}(u,v;s).
\end{equation*}
Let $N$ be an arbitrary natural number.
By the Taylor formula,
\begin{equation}\label{eqn:4.7}
\begin{split}
\psi(u,v;s)
&=\sum_{(\a,\b)\in \{0,\ldots,N\}^2}
\frac{\psi^{\langle \a, \b \rangle}(0,0;s)}{\a ! \b !}
u^{\a}v^{\b}
+
\sum_{\a=0}^{N}u^{\a}v^{N+1}
\tilde{A}_{\alpha}^{(N)}(v;s)\\
&\quad\quad\quad
+\sum_{\b=0}^{N}u^{N+1}v^{\b}
\tilde{B}_{\beta}^{(N)}(u;s)
+u^{N+1}v^{N+1}\tilde{C}^{(N)}(u,v;s),
\end{split}
\end{equation}
where
\begin{equation}\label{eqn:4.8}
\begin{split}
&\tilde{A}_{\alpha}^{(N)}(v;s)
:=\frac{1}{\a ! N!}\int_{0}^{1}(1-t)^{N}
\psi^{\langle \a,N+1 \rangle}(0,tv;s)dt
 \mbox{\quad for } \a\in\{0,\ldots, N\}, \\
&\tilde{B}_{\beta}^{(N)}(u;s)
:=\frac{1}{\b ! N!}
\int_{0}^{1}(1-t)^{N}
\psi^{\langle N+1, \b \rangle}(tu,0;s)dt
\mbox{\quad for }  \b\in \{0,\ldots,N\},\\
&\tilde{C}^{(N)}(u,v;s)
:=\frac{1}{(N!)^2}\int_{0}^{1}\int_{0}^{1}
(1-t_1)^{N}(1-t_2)^{N}
\psi^{\langle N+1, N+1 \rangle}(t_1u,t_2v;s)
dt_1 dt_2.
\end{split}
\end{equation}
Note that
\begin{itemize}
\item $\tilde{A}_{\alpha}^{(N)}(\cdot;s)$, 
$\tilde{B}_{\beta}^{(N)}(\cdot;s)$, 
$\tilde{C}^{(N)}(\cdot;s)$ 
are smooth functions for each $s\in\C$.
\item $\tilde{A}_{\alpha}^{(N)}(v;\cdot)$, 
$\tilde{B}_{\beta}^{(N)}(u;\cdot)$, 
$\tilde{C}^{(N)}(u,v;\cdot)$ 
are entire functions for each $(u,v)\in D$.
\end{itemize}
Substituting (\ref{eqn:4.7}) into (\ref{eqn:4.5}), 
we have 
\begin{equation}\label{eqn:4.9}
\begin{split}
H(s)=&\sum_{(\a,\b)\in \{0,\ldots,N\}^2}
\frac{\psi^{\langle \a, \b \rangle}(0,0;s)}{\a ! \b !}H_{\a, \b}(s)\\
&\quad+
\sum_{\alpha=0}^N A_{\alpha}^{(N)}(v;s)+
\sum_{\beta=0}^N B_{\beta}^{(N)}(v;s)+
C^{(N)}(s)
\end{split}
\end{equation}
with
\begin{equation*}
\begin{split}
&H_{\a, \b}(s)=
\int_{D}u^{as+\a}v^{bs+\b}dudv 
\mbox{\quad for } (\a, \b)\in \{0,\ldots, N\}^2, \\
&A_{\alpha}^{(N)}(s)=
\int_{D}u^{as+\a}v^{bs+N+1}\tilde{A}_{\alpha}^{(N)}(v;s)dudv 
 \mbox{\quad for } \a\in\{0,\ldots, N\},\\
&B_{\beta}^{(N)}(s)=
\int_{D}u^{as+N+1}v^{bs+\b}\tilde{B}_{\beta}^{(N)}(u;s)dudv  
\mbox{\quad for }  \b\in \{0,\ldots,N\},\\
&C^{(N)}(s)=
\int_{D}u^{as+N+1}v^{bs+N+1}\tilde{C}^{(N)}(u,v;s)dudv.
\end{split}
\end{equation*}


Now let us consider the meromorphy of the above 
integrals. 

\vspace{.5 em}

\underline{The integral $H_{\a,\b}(s)$.}\quad 

A simple computation gives that 
\begin{equation*}
\begin{split}
H_{\a,\b}(s)&=
\int_{0}^{r}v^{bs+\b}
\left(\int_{v^p}^{R}u^{as+\a}du\right)dv
\\
&=
\frac{1}{as+\alpha+1}
\left( 
\frac{R^{as+\alpha+1}r^{bs+\beta+1}}{bs+\beta+1}
-\frac{r^{(ap+b)s+\a p+\b+p+1}}
{(ap+b)s+\a p+\b +p+1}
\right),
\end{split}
\end{equation*}
which implies that 
every $H_{\a,\b}(s)$ becomes a meromorphic function 
on the whole complex plane
and its poles are contained in the set
\begin{equation}\label{eqn:4.10}
\left\{-\frac{j}{a},-\frac{k}{b},-\frac{p+l}{ap+b}:
j,k,l\in\N \right\}.
\end{equation}

\vspace{.5 em}

\underline{The integral $A_{\a}^{(N)}(s)$.}\quad 

A simple computation gives that 
\begin{eqnarray}
&&A_{\alpha}^{(N)}(s)=
\int_0^r\left(\int_{v^p}^{R}
u^{as+\a}du
\right)
v^{bs+N+1}\tilde{A}_{\alpha}^{(N)}(v;s)dv \nonumber\\
&&
\begin{split}
&\quad\quad\quad
=\frac{1}{as+\alpha+1}
\left(
R^{as+\alpha+1}\int_{0}^{r}v^{bs+N+1}
\tilde{A}_{\alpha}^{(N)}(v;s)dv
\right. \\ 
&\quad\quad\quad\quad\quad\quad \quad\quad
\left.
-\int_{0}^{r}v^{(ap+b)s+\alpha p+p+N+1}
\tilde{A}_{\alpha}^{(N)}(v;s)dv
\right). \label{eqn:4.11}
\end{split}
\end{eqnarray}
Lemma~4.1 (i) implies that
the first (resp. the second) integral in 
(\ref{eqn:4.11})
becomes a holomorphic function on the half-plane
${\rm Re}(s)>-(N+2)/b$ 
(resp. ${\rm Re}(s)>-(\alpha p+p+N+2)/(ap+b)$).
Thus, 
$A_{\a}^{(N)}(s)$ can be analytically continued 
as a meromorphic function to the half-plane 
${\rm Re}(s)>\max\{-(N+2)/b,-(\alpha p+p+N+2)/(ap+b)\}$ 
and its poles are contained in the set
\begin{equation}\label{eqn:4.12}
\left\{-\frac{j}{a}:j\in\N \right\}.
\end{equation}
(When $N$ is sufficiently large, 
the above maximum is $-(\alpha p+p+N+2)/(ap+b)$.)

\vspace{.5 em}

\underline{The integral $B_{\beta}^{(N)}(s)$.}\quad 

A simple computation gives that 
\begin{eqnarray}
&&B_{\beta}^{(N)}(s)=
\left(
\int_{0}^{r^p}\int_0^{u^{1/p}}+
\int_{r^p}^{R}\int_0^{r}
\right)
u^{as+N+1}v^{bs+\beta}
\tilde{B}_{\beta}^{(N)}(u;s)dvdu \nonumber\\
&&\quad\quad\quad
\begin{split}
&=\frac{1}{bs+\beta+1}
\left(
\int_{0}^{r^p}
u^{\frac{1}{p}\{(ap+b)s+pN+p+\beta+1\}}
\tilde{B}_{\beta}^{(N)}(u;s)du \right.
\\
&\quad\quad\quad\quad\quad\quad\quad\quad\quad\quad 
\left. +
r^{bs+\beta+1}\int_{r^p}^{R}
u^{as+N+1}
\tilde{B}_{\beta}^{(N)}(u;s)du
\right). \label{eqn:4.13}
\end{split}
\end{eqnarray}
Lemma~4.1 (i) implies that
the first integral in (\ref{eqn:4.13})
becomes a holomorphic function 
on the half-plane ${\rm Re}(s)>-(pN+2p+\beta+1)/(ap+b)$.
Moreover, it is easy to check that the second integral
is an entire function.  
Hence, $B_{\beta}^{(N)}(s)$ 
can be analytically continued 
as a meromorphic function to the half-plane 
${\rm Re}(s)>-(pN+2p+\beta+1)/(ap+b)$
and its poles are contained in the set
\begin{equation}\label{eqn:4.14}
\left\{-\frac{k}{b}:k\in\N \right\}.
\end{equation}

\vspace{.5 em}

\underline{The integral $C^{(N)}(s)$.}\quad 

It follows from the proof of (i) in this lemma that
the integral $C^{(N)}(s)$ converges
on the half-plane 
${\rm Re}(s)>\max\{-(N+2)/a,-(N+2)/b\}=-(N+2)/b$,
which implies that
$C^{(N)}(s)$ can be analytically continued as a holomorphic function 
there.

\vspace{.5 em}

From the above, letting $N$ to infinity 
in (\ref{eqn:4.9}), 
we can see that $H(s)$ becomes a meromorphic function 
on the whole complex plane and that
the poles of $H$ is contained in the set (\ref{eqn:4.6})
from (\ref{eqn:4.10}), (\ref{eqn:4.12}), (\ref{eqn:4.14}).
\end{proof}

%

\subsection{A van der Corput-type lemma}

\begin{lemma}[\cite{Gre06}]
Let $f$ be a $C^{k}$ function on an interval $I$ in $\R$.
If 
$|f^{(k)}|>\eta$
on $I$, then
for $\sigma\in (-1/k,0)$ there is a positive constant 
${\mathcal C}(\sigma,k)$ depending only on $\sigma$ and $k$
such that
\begin{equation}\label{eqn:4.15}
\int_{I}|f(x)|^{\sigma}dx<
{\mathcal C}(\sigma,k)
\eta^{\sigma}|I|^{1+k\sigma},
\end{equation}
where $|I|$ is the length of $I$.
\end{lemma}

The above van der Corput-type lemma plays the most important role 
in our analysis. 
This lemma has been shown in \cite{Gre06}, but
we give a proof with more detailed explanation
for convenience of readers.

\begin{proof}
Let $F(\lambda)$ be the distribution function of $|f|$, that is,
$F(\lambda)$ is the length of the set $\{x\in I:|f(x)|\leq \lambda\}$.
By using $F(\lambda)$,
the integral in \eqref{eqn:4.15} is represented as
\begin{equation*}
\int_I|f(x)|^{\sigma} dx=-\sigma\int_0^{\infty}\lambda^{\sigma-1}F(\lambda)d\lambda.
\end{equation*}
Indeed, let
$E=\{(x,\lambda)\in I\times (0,\infty):|f(x)|\leq \lambda\}$ and 
$\mathbf{1}_E$ be the characteristic function of $E$,
then
\begin{equation}\label{eqn:4.16}
\begin{split}
\int_I|f(x)|^{\sigma} dx
&=-\sigma\int_I\left(\int_{|f(x)|}^{\infty}\lambda^{\sigma-1}d\lambda\right) dx\\
&=-\sigma\int_{I}\left(\int_{0}^{\infty}\lambda^{\sigma-1}
\mathbf{1}_{E}(x,\lambda)d\lambda\right) dx\\
&=-\sigma\int_{0}^{\infty}\lambda^{\sigma-1}\left(
\int_{I}\mathbf{1}_{E}(x,\lambda)dx \right)d\lambda\\
&=-\sigma\int_0^{\infty}\lambda^{\sigma-1}F(\lambda)d\lambda.
\end{split}
\end{equation}
We remark that the third equality in (\ref{eqn:4.16}) is given by
Tonelli's theorem for nonnegative measurable functions.
Here, we decompose the last integral in (\ref{eqn:4.16}) as follows.
\begin{equation*}
\begin{split}
\int_{0}^{\infty}\lambda^{\sigma-1}F(\lambda)d\lambda
&=\int_{0}^{\lambda_0}\lambda^{\sigma-1}F(\lambda)d\lambda+
\int_{\lambda_0}^{\infty}\lambda^{\sigma-1}F(\lambda)d\lambda\\
&=:J_1+J_2,
\end{split}
\end{equation*}
where $\lambda_0=\eta|I|^{k}$.
From Lemma 3.3 in \cite{Chr85},
we have
\begin{equation}\label{eqn:4.17}
F(\lambda)\leq C \frac{\lambda^{1/k}}{\eta^{1/k}},
\end{equation}
where $C$ is a positive constant depending only on $k$.
Easy computation with
the inequality (\ref{eqn:4.17}) gives 
$J_1\leq (Ck/(1+k\sigma))\eta^{\sigma}|I|^{1+k\sigma}$.
On the other hand,
we have
$J_2\leq (-1/\sigma)\eta^{\sigma}|I|^{1+k\sigma}$
since $F(\lambda)\leq |I|$.
The proof is completed.

\end{proof}


\section{Properties of $f$ in (\ref{eqn:1.2})}

Let $f$ be a smooth function expressed as in the case (C) in Lemma~3.1
on a small open neighborhood $U$ of the origin.

\subsection{The zero-variety $V(f)$}

In order to understand the analytic continuation
of local zeta functions, it is essentially important to 
understand geometric 
properties of the zero-variety:
\begin{equation}\label{eqn:5.1}
V(f)=\{(x,y)\in U:f(x,y)=0\}.
\end{equation}
It follows from an important factorization formula of Rychkov \cite{Ryc01} 
that $f$ can be expressed as follows.
\begin{lemma}
There exists a small open neighborhood $U$ of the origin 
such that 
\begin{equation*}\label{eqn:}
f(x,y)=\tilde{v}(x,y)y^b\prod_{j=1}^a (x-\phi_j(y)) \,\,\,
\mbox{ on $U$},
\end{equation*}
where $\phi_j$ are complex-valued continuous functions defined
near the origin satisfying that 
$\phi_j(y)=O(y^N)$ as $y\to 0$ for any $N\in\N$ and 
$\tilde{v}$ is a smooth function defined on $U$ with 
$\tilde{v}(0,0)>0$.
\end{lemma}
\begin{proof}
This is an easy case of Proposition~2.1 in \cite{Ryc01}.
\end{proof}

From the above lemma, 
the zero variety $V(f)$ is composed 
by at most $a+1$ components: 
\begin{equation*}
\begin{split}
&Z_0:=\{(x,y)\in U:y=0\}, \\
&Z_j:=\{(x,y)\in U:x=\phi_j(y)\} \,\,\text{ for $j=1,\ldots,a$.}
\end{split}
\end{equation*}
When $\phi_j$ takes non-real values, 
$Z_j$ may be equal to $\{(0,0)\}$ but
they are realized in the complex space.
It is possible  
that $Z_j=Z_k$ for some $j,k$.

Roughly speaking, 
for the meromorphic extension of $Z_f(\varphi)$, 
the variety $Z_0$ gives {\it good} influence, 
while the varieties $Z_j$ 
with $\phi_j\not\equiv 0$ for $j=1,\ldots,a$ 
give {\it bad} one.
Therefore, in the analysis of $Z_f(\varphi)$, 
the integral region is divided into two parts: 
one is avoided from the bad varieties as large as possible, 
while the other is its complement. 
Of course, their shapes must be useful for the computation.
Actually, 
we use the two regions with a parameter $m\in\N$:
\begin{equation}\label{eqn:5.2}
\begin{split}
U_1^{(m)}&=\{(x,y)\in U: x>y^{m},\,\, 0\leq y \leq r_m\},\\
U_2^{(m)}&=\{(x,y)\in U: 0<x \leq y^{m}, \,\,  0\leq y \leq r_m\},
\end{split}
\end{equation}
where $r_m>0$ will be appropriately decided later 
in (\ref{eqn:5.5}).

\subsection{Two expressions of $f$}

In order to understand 
important properties of $f$,
we express $f$ by using the following two functions:
$F:U\to\R$ and $\tilde{f}:U\setminus\{x=0\}\to\R$
defined by 
\begin{equation}\label{eqn:5.3}
\begin{split}
&F(x,y)=v(x,y)x^a+\sum_{j=0}^{a-1}x^j \tilde{h}_j(y),\\
&\tilde{f}(x,y)=v(x,y)+\sum_{j=1}^{a}\frac{\tilde{h}_{a-j}(y)}{x^j},
\end{split}
\end{equation}
where 
$\tilde{h}_j(y):=h_{j}(y)/y^b$ 
for $j=0,\ldots,a-1$. 
Note that each $\tilde{h}_j$ is flat at the origin. 
Then $f$ can be expressed as  
\begin{equation}\label{eqn:5.4}
\begin{split}
f(x,y)&=y^b F(x,y) \quad \mbox{ on $U$,}\\
f(x,y)&=x^a y^b\tilde{f}(x,y) \quad \mbox{ on $U\setminus\{x=0\}$.}
\end{split}
\end{equation}

In order to investigate $f$, we use 
$\tilde{f}$ on $U_1^{(m)}$ and
$F$ on $U_2^{(m)}$.


\subsection{Properties of $\tilde{f}$}

Let $\phi$ be a function defined near the origin as 
$$\phi(y)=\max\{|\phi_j(y)|:j=1,\ldots,a\},$$
where $\phi_j$ is as in Lemma 5.1.
Since $\phi_j(y)=O(y^N)$ as $y\to 0$ for any $N\in\N$,
for each $m\in\N$ there exists a positive number $r_m$
such that 
\begin{equation}\label{eqn:5.5}
\phi(y)\leq \frac{1}{2}y^m \quad \text{ for $y\in[0,r_m]$}.
\end{equation}
We take the value of $r_m$ in (\ref{eqn:5.2}) such that
(\ref{eqn:5.5}) holds.

\begin{lemma}
The function $\tilde{f}$ satisfies the following properties.
\begin{enumerate}
\item 
There exists a positive number $c$ independent of $m$ 
such that  $\tilde{f}(x,y)\geq c$ on $U_1^{(m)}$.
\item 
$\tilde{f}(x,y)$ can be smoothly extended  
to $\overline{U_1^{(m)}}$.
\end{enumerate}
\end{lemma}

\begin{proof}

(i)\quad 
From Lemma~5.1, $\tilde{f}$ takes the following 
form on $U_1^{(m)}$:
\begin{equation}\label{eqn:5.6}
\tilde{f}(x,y)=
\tilde{v}(x,y)\prod_{j=1}^a 
\left(
1-\frac{\phi_j(y)}{x}
\right).
\end{equation}
From (\ref{eqn:5.5}), we easily see that
\begin{equation*}
\left|\frac{\phi_j(y)}{x}\right|\leq
\frac{\phi(y)}{x}\leq 
\frac{1}{2}\frac{y^m}{x}\leq
\frac{1}{2}
\quad \mbox{for $(x,y)\in U_1^{(m)}$},
\end{equation*}
which implies that
\begin{equation*}
\begin{split}
\tilde{f}(x,y) \geq \frac{\tilde{v}(0,0)}{2^a}
\quad \mbox{for $(x,y)\in U_1^{(m)}$}.
\end{split}
\end{equation*}

(ii)\quad
Since $\tilde{f}$ is a smooth function
on $U\setminus\{x=0\}$, 
it suffices to show that
all partial derivatives of 
$\tilde{f}$ can be continuously extended to 
the origin. 
Moreover, from the second equation in (\ref{eqn:5.3}),
it suffices to show that all partial derivatives of 
$\psi_j(x,y):=\tilde{h}_{a-j}(y)/x^j$ can be continuously
extended to the origin.

Let $(\alpha,\beta)\in\Z_+^2$ be arbitrarily given.
An easy computation gives that
\begin{equation*}
\frac{\d ^{\a+\b} \psi_j}{\d x^{\a} \d y^{\b}}
=(-1)^{\alpha} \frac{(j+\alpha-1)!}{(j-1)!} 
\frac{\tilde{h}_{a-j}^{(\beta)}(y)}{x^{j+\alpha}}
\end{equation*}
where $\tilde{h}_{a-j}^{(\beta)}$ 
is the $\b$-th derivative of $\tilde{h}_{a-j}$.
Since $\tilde{h}_{a-j}^{(\beta)}$ is flat at the origin,
there exists $r_{m,\alpha,\beta}>0$ such that
$
|\tilde{h}_{a-j}^{(\beta)}(y)|\leq 
\frac{(j-1)!}{(j+\alpha-1)!}y^{m(j+\alpha+1)}$
for $y\in [0,r_{m,\alpha,\beta}]$ and $j=1,\ldots,a$.
Considering the shape of the set $U_1^{(m)}$, 
we have
$$
\left|
\frac{\d^{\a+\b} \psi_j}{\d x^{\a} \d y^{\b}}\right|
\leq x \quad 
\mbox{ for 
$U_1^{(m)}\cap \{0<y\leq r_{m,\alpha,\beta}\}$ 
and $j=1,\ldots,a$.}
$$
Therefore, 
\begin{equation*}
\lim_{U_1^{(m)}\ni (x,y)\to (0,0)}
\frac{\partial^{\a+\b}\psi_j}{\d x^{\a} \d y^{\b}}(x,y)=0
\quad
\mbox{ for $j=1,\ldots,a$}.
\end{equation*}
In particular,  
$\frac{\d ^{\a+\b}\psi_j}{\d x^{\a} \d y^{\b}}$ 
can be continuous up to
the origin for $j=1,\ldots, a$.
\end{proof}


\subsection{Properties of $F$}

When a van der Corput-type lemma in Lemma~4.5 is applied
in the next section, 
the following lemma is important.

\begin{lemma}
There exist $R>0$ and $\mu>0$ such that 
$$
\frac{\partial^a}{\partial x^a}F(x,y)
\geq \mu 
\quad \mbox{ on $[-R,R]^2$.}
$$
\end{lemma}

\begin{proof}
A direct computation gives 
$$
\frac{\partial^a}{\partial x^a}F(0,0)=
a! v(0,0) \,\,(>0),
$$
which implies the lemma.
\end{proof}

\section{Proof of Theorem~3.3}

Let $f$ be expressed as in the case (C) 
in Lemma~3.1 on a small open neighborhood $U$
of the origin. 
Let $U_j^{(m)}$ ($j=1,2$) be as in (\ref{eqn:5.2}) and 
let $r_m$ be a positive constant determined by (\ref{eqn:5.5}). 
Let $a>0$.

\subsection{A decomposition of $Z_f(\varphi)(s)$}

Using the orthant decomposition, we have
\begin{equation}\label{eqn:6.1}
Z_f(\varphi)(s)=
\sum_{\theta\in\{1,-1\}^2}
\tilde{Z}_{f_\theta}(\varphi_{\theta})(s),
\end{equation}
where
\begin{equation}\label{eqn:6.2}
\tilde{Z}_f(\varphi)(s)=\int_{\R_+^2}\left|f(x,y)\right|^s \varphi(x,y)dxdy,\\
\end{equation}
$f_{\theta}(x,y)=f(\theta_1 x, \theta_2 y)$
and
$\varphi_{\theta}(x,y)=\varphi(\theta_1 x, \theta_2 y)$
with $\theta=(\theta_1,\theta_2)$.
In order to prove the theorem,
it suffices to consider the integral
$\tilde{Z}_f(\varphi)(s)$.

Now, let us decompose $\tilde{Z}_f(\varphi)(s)$ as
\begin{equation}\label{eqn:6.3}
\tilde{Z}_f(\varphi)(s)=
I_1^{(m)}(s)+I_2^{(m)}(s)+J^{(m)}(s)
\end{equation}
with
\begin{equation}\label{eqn:6.4}
\begin{split}
I_j^{(m)}(s)
&=\int_{U_j^{(m)}}
\left|f(x,y)\right|^s 
\varphi(x,y)\chi_m(y)dxdy \quad 
\mbox{for $j=1,2$},\\
J^{(m)}(s)
&=\int_{\R_+^2}\left|f(x,y)\right|^s 
\varphi(x,y)(1-\chi_m(y))dxdy,\\
\end{split}
\end{equation}
where
$\chi_m:\R\to[0,1]$ is a cut-off function satisfying that
$\chi_m(y)= 1$ if $|y|\leq r_m/2$ and 
$\chi_m(y)= 0$ if $|y|\geq r_m$.

The following lemma will play a useful role 
in the analysis of the integral $I_1^{(m)}(s)$. 
\begin{lemma}
Let $\Psi:U_1^{(m)}\times \C\to\C$ be defined by 
\begin{equation*}
\Psi(x,y;s)=
\tilde{f}(x,y)^s\varphi(x,y)\chi_m(y), 
\end{equation*}
where $\tilde{f}$ is as in (\ref{eqn:5.3}).
Then we have
\begin{enumerate}
\item $\Psi(\cdot;s)$ can be smoothly extended to 
$\overline{U_1^{(m)}}$ for all $s\in\C$.
\item 
$\dfrac{\d^{\alpha+\beta}\Psi}{\d x^{\alpha}\d y^{\beta}}(x,y;\cdot)$ 
is an entire function 
for all $(x,y)\in U_1^{(m)}$ and $(\alpha, \beta)\in \Z_+^2$.
\end{enumerate}
\end{lemma}

\begin{proof}
Since $\varphi(x,y)\chi_m(y)$ 
does not give any essential influence on the properties (i), (ii), 
it suffices to consider the function 
$\tilde{f}(x,y)^s$. 
On the domain 
$U_1^{(m)}\setminus \{(x,y):\tilde{f}(x,y)=0\}$, 
every partial derivative of $\tilde{f}(x,y)^s$
with respect to $x,y$ can be expressed as the sum of
$s(s-1)\cdots (s-k+1) \tilde{f}(x,y)^{s-k}$ for $k\in\N$
multiplied by polynomials of the partial derivatives
of $\tilde{f}(x,y)$ with respect to $x,y$. 
Applying Lemma~5.2 to this expression, 
we can see that 
$\tilde{f}(x,y)^s$ has the properties in (i), (ii). 
\end{proof}

\subsection{Meromorphic continuation of associated integrals}

In order to prove Theorem~3.3,
it suffices to show the following.

\begin{lemma}
Let $m\in\N$.
If the support of $\varphi$ is contained in $[-R,R]^2$ 
where $R>0$ is as in Lemma 5.3,
then the following hold:
\begin{enumerate}
\item
$I_1^{(m)}(s)$ 
can be analytically continued as a meromorphic function
to the whole complex plane. 
Moreover,
its poles are contained in the set 
\begin{equation*}
\left\{
-\frac{j}{a},-\frac{k}{b},-\frac{m+l}{am+b}:j,k,l\in \N
\right\}.
\end{equation*}
\item
$I_2^{(m)}(s)$ can be holomorphically continued as a holomorphic function
to the half-plane ${\rm Re}(s)>-(m+1)/(am+b)\, (\geq-1/a) $.
\item
$J^{(m)}(s)$ can be holomorphically continued as a holomorphic function
to the half-plane
${\rm Re}(s)>-1/a$.
\end{enumerate}
\end{lemma}

\begin{remark}
The restriction of the set of candidate poles in (i) 
to the region  ${\rm Re}(s)>-(m+1)/(am+b)$
is contained in 
$\left\{
-k/b:k \in \N \,\,{\rm with}\,\, k<b/a
\right\}$.
\end{remark}

\begin{proof}

(i)\quad
From (\ref{eqn:5.4}), 
$I_1^{(m)}(s)$ can be expressed as 
\begin{equation*}
\begin{split}
I_1^{(m)}(s)&=\int
_{U_1^{(m)}}
x^{as}y^{bs}\Psi(x,y;s)dxdy.
\end{split}
\end{equation*}
Since Lemma~6.1 implies that
$\Psi$ satisfies the same properties as those of 
$\psi$ in Lemma 4.3, 
the integral $I_1^{(m)}(s)$
can be analytically continued 
as a meromorphic function to the whole complex plane
and, moreover, its poles are contained in the set
\begin{equation*}
\left\{-\frac{j}{a},-\frac{k}{b},-\frac{m+l}{am+b}:
j,k,l\in \N\right\}.
\end{equation*}


\vspace{.5 em}

(ii) \quad
From (\ref{eqn:5.4}), 
$I_2^{(m)}(s)$ can be expressed as
\begin{equation*}
I_2^{(m)}(s)
= 
\int_{U_2^{(m)}}
y^{bs}
\left|F(x,y)
\right|^{s}
\varphi(x,y)\chi_m(y)dxdy.
\end{equation*}
It is easy to see that
\begin{equation}\label{eqn:6.5}
|I_2^{(m)}(s)|
\leq C_m 
\int_{0}^{r_m}y^{b{\rm Re}(s)}
\left(
\int_{0}^{y^{m}}
\left|F(x,y)
\right|^{{\rm Re}(s)}dx
\right)
dy,
\end{equation}
where $C_m=\sup_{(x,y)\in U_2^{(m)}}(|\varphi(x,y) \chi_m(y)|)$.
Since  
Lemma 4.5 can be applied to the integral 
with respect to the variable $x$ 
in (\ref{eqn:6.5}) from Lemma~5.3, 
if ${\rm Re}(s)>-1/a$,
then
\begin{equation}\label{eqn:6.6}
\begin{split}
|I_2^{(m)}(s)|
&< C_m {\mathcal C}({\rm Re}(s),a)\mu^{{\rm Re}(s)}
\int_{0}^{r_m}
y^{b{\rm Re}(s)}
\left(y^{m}\right)^{1+a{\rm Re}(s)}
dy\\
&= C_m {\mathcal C}({\rm Re}(s),a)\mu^{{\rm Re}(s)}
\int_{0}^{r_m}
y^{(am+b){\rm Re}(s)+m}dy,
\end{split}
\end{equation}
where ${\mathcal C}(\cdot,\cdot)$
is as in Lemma~4.5 and $\mu$ is as in Lemma~5.3.  
The last integral in (\ref{eqn:6.6}) converges
on the half-plane ${\rm Re}(s)>-(m+1)/(am+b)$, 
on which
$I_2^{(m)}(s)$ becomes a holomorphic function. 
We remark that 
$-(m+1)/(am+b)\geq -1/a$ holds
for all $m\in\N$.

\vspace{.5 em}

(iii)\quad
In a similar fashion to the case of 
integral $I_2^{(m)}(s)$, 
we have
\begin{equation*}
\begin{split}
|J^{(m)}(s)|\leq \tilde{C}_m
\int_{r_m/2}^{R}
y^{b{\rm Re}(s)}
\left(
\int_{0}^{R} |F(x,y)|^{{\rm Re}(s)}dx
\right)
dy,
\end{split}
\end{equation*}
where 
$\tilde{C}_m:=\sup_{(x,y)\in[0,R]\times [r_m/2,R]}
(|\varphi(x,y) (1-\chi_m(x))|).$
Applying Lemma 4.5, 
we have that
if 
${\rm Re}(s)>-1/a$,
then
\begin{equation*}
\begin{split}
|J^{(m)}(s)|
\leq  \tilde{C}_m {\mathcal C}({\rm Re}(s),a)\mu^{{\rm Re}(s)}
R^{1+a{\rm Re}(s)}
\int_{r_m/2}^{R}y^{b{\rm Re}(s)}dy.
\end{split}
\end{equation*}
Since the above integral converges for any $s\in \C$, 
$J^{(m)}(s)$
can be analytically continued as a holomorphic function
to the half-plane ${\rm Re}(s)>-1/a$.
\end{proof}

\subsection{Proof of Theorem 3.3}

From (\ref{eqn:6.1}), (\ref{eqn:6.2}), 
(\ref{eqn:6.3}), (\ref{eqn:6.4}), 
Lemma 6.2 gives Theorem 3.3 by 
letting $m$ to infinity.


\appendix

\section{Newton polyhedra}
 
After giving the definitions of
Newton polyhedra, Newton distances and
adapted coordinates and briefly explaining
their properties, 
we observe our study from these points of view.  

Let $f$ be a real-valued smooth function defined 
on an open neighborhood of the origin in $\R^2$. 

\subsection{Newton polyhedra}

The Taylor series of $f$ at the origin is   
\begin{equation}\label{eqn:a.1}
f(x,y)\sim \sum_{(\alpha,\beta)\in{\Z}_+^2} 
c_{\alpha\beta}x^{\alpha}y^{\beta} 
\quad \mbox{ with $c_{\alpha\beta}=
\dfrac{1}{\alpha!\beta!}
\dfrac{\partial^{\alpha+\beta} f}
{\partial x^{\alpha}\partial y^{\beta}}(0,0)$}.
\end{equation}
The {\it Newton polyhedron} of $f$
is the integral polyhedron: 
$$
\Gamma_+(f)=
\mbox{the convex hull of the set 
$\bigcup \{(\a,\beta)+\R_+^2:
c_{\alpha\beta}\neq 0\}$ in $\R_+^2$}
$$
(i.e., the intersection 
of all convex sets 
which contain $\bigcup \{(\a,\beta)+\R_+^2:c_{\alpha\beta}\neq 0\}$). 
Note that 
the flatness of $f$ at the origin is the equivalent 
to the condition:
$\Gamma_+(f)=\emptyset$.
We say that $f$ is {\it convenient} if the Newton polyhedron 
of $f$ intersects every coordinate axis.

\subsection{Newton distances}

We assume that $f$ is nonflat.
The {\it Newton distance} $d(f)$ of $f$ is defined by 
\begin{equation*}
d(f)=\inf\{\alpha>0:(\alpha,\alpha)\in\Gamma_+(f)\}.
\end{equation*}
The minimal face of $\Gamma_+(f)$ containing 
the point $(d(f),d(f))$ is called {\it principal face}
of $\Gamma_+(f)$.
Since the Newton distance depends on the coordinates system 
$(x,y)$ on which $f$ is defined, 
it is sometimes denoted by $d_{(x,y)}(f)$.

\subsection{Adapted coordinates}

A given coordinate system $(x,y)$ is said to be {\it adapted} to
$f$, 
if the equality 
$$d_{(x,y)}(f)=\sup_{(u,v)}\{d_{(u,v)}(f)\}$$
holds, 
where the supremum is taken over all local smooth coordinate 
systems $(u,v)$ at the origin. 
The existence of adapted coordinates 
is shown in \cite{Var76}, \cite{PSS99}, \cite{IkM11tams}, etc.
Furthermore, 
useful necessary and sufficient conditions 
for the adaptedness have been obtained.
It is known in \cite{IkM11tams} that 
if the principal face of $\Gamma_+(f)$ is a noncompact face or a vertex
of $\Gamma_+(f)$, 
then the respective coordinate is adapted to $f$. 
It follows from this fact that the function in 
(\ref{eqn:1.2}) is 
defined in an adapted coordinate.

\begin{remark}
The existence of adapted coordinates is not obvious. 
The definition of 
the adapted coordinate can be directly 
generalized in higher dimensional case.
In three-dimensional case, 
it is known in \cite{Var76} that 
there exists a function admitting
no adapted coordinate.  
\end{remark}

\subsection{The $\gamma$-part and the class $\hat{\mathcal E}(U)$ }

Any line in $\R^2$ can be expressed 
by using some pair $(a,b;l)\in\R^2\times\R$
as 
\begin{equation*}
L(a,b;l):=
\{(\alpha,\beta)\in\R^2:a\alpha+b\beta=l
\}.
\end{equation*} 
For any edge $\gamma\subset\Gamma_+(f)$, 
there exists a unique pair $(a,b;l)\in\Z_+^2\times\Z_+$
with gcd$(a,b)=1$ such that 
\begin{equation}\label{eqn:a.2}
\gamma=L(a,b;l)\cap\Gamma_+(f).
\end{equation}
For a given face $\gamma$ of $\Gamma_+(f)$,
we say that $f$ {\it admits the $\gamma$-part}
on an open neighborhood 
$U$ of the origin 
if for any $(x,y)\in U$, the limit:
\begin{equation*}
\lim_{t\to 0}
\frac{f(t^{a} x, t^{b} y)}{t^l}
\end{equation*}
exists for the pair $(a,b;l)$ 
defining $\gamma$ through (\ref{eqn:a.2}). 
This process produces the function on $U$, 
which is called the $\gamma$-{\it part of} $f$ and
denoted by $f_{\gamma}$.
When a face $\gamma$ is compact,
$f$ always admits the $\gamma$-part, 
which can be simply expressed on $U$ as  
$$
f_{\gamma}(x,y)=\sum_{(\alpha,\beta)\in\gamma\cap\Z_+^2} 
c_{\alpha\beta}x^{\alpha}y^{\beta},
$$
where $c_{\alpha\beta}$ are the same as in (\ref{eqn:a.1}).

The class $\hat{\mathcal E}(U)$ 
consists of the smooth functions admitting 
the $\gamma$-part for all the edges $\gamma$ 
of $\Gamma_+(f)$.
This class contains many kinds of 
smooth functions (see \cite{KaN16jmst}).
\begin{itemize}
\item
Every real analytic function 
belongs to $\hat{\mathcal E}(U)$. 
\item 
Every convenient function belongs to $\hat{\mathcal E}(U)$. 
\item
The Denjoy-Carleman (quasianalytic) classes
are contained in $\hat{\mathcal E}(U)$.  
\end{itemize}
It is shown in \cite{KaN16jmst}
that if $f$ belongs to the class $\hat{\mathcal E}(U)$,
then the results of Varchenko \cite{Var76} 
concerning the real analytic case
can be directly generalized.

\subsection{The case of $f$ in (\ref{eqn:1.2})}

Now let us consider the case when $f$ is as in (\ref{eqn:1.2})
with $a,b\in\N$. 
As mentioned in the beginning of Section 3, 
the Newton data of $f$ can be
easily obtained. 
Note that $f$ is not convenient. 
Moreover, the Newton polyhedron $\Gamma_+(f)$ 
has the two noncompact edges:
\begin{equation*}
\gamma_1=\{(\alpha,\beta):\alpha\geq a, \, \beta=b\},\quad
\gamma_2=\{(\alpha,\beta):\alpha=a, \, \beta\geq b\}.
\end{equation*} 

The classification in Lemma~3.1, which is obtained by 
the Taylor formula, can be expressed 
by using the admission of the $\gamma_1$,  $\gamma_2$-parts.
\begin{enumerate}
\item[(A)]
$f$ admits both the $\gamma_1$-part and the $\gamma_2$-part, i.e.,
$f\in\hat{\mathcal E}(U)$.  
\item[(B)]
$f$ admits the $\gamma_2$-part but it does not admit the $\gamma_1$-part. 
\item[(C)]
$f$ admits the $\gamma_1$-part but it does not admit the $\gamma_2$-part. 
\item[(D)]
$f$ admits neither the $\gamma_1$-part nor the $\gamma_2$-part. 
\end{enumerate}
From \cite{KaN16jmst}, 
the case (A) can be easily treated in a similar fashion 
to the real analytic case and, in particular, 
$m_0(f)=\infty$ is shown.
In the other cases, since 
$f$ does not belong to the class $\hat{\mathcal E}(U)$,
the result in \cite{KaN16jmst} cannot be applied.

\bigskip

{\it Acknowledgments.} \quad
The authors greatly appreciate that the referee 
carefully read the first version of this paper and 
gave many valuable comments.  
This work was supported by 
JSPS KAKENHI Grant Numbers 
JP15K04932,  JP19K14563,  JP15H02057.

\end{document}